\documentclass[twoside]{article}
\title{\bf Generalized Bessel and Frame Measures}
\author{Fariba Zeinal Zadeh Farhadi$^{1}$, Mohammad Sadegh Asgari$^{2}$,\\ 
Mohammad Reza Mardanbeigi$^{3^*}$ and Mahdi Azhini$^{4}$}
\date{}

\usepackage{amsmath,amsthm,amssymb,amscd,amsfonts}
\usepackage[bookmarksnumbered, colorlinks, plainpages]{hyperref}
\theoremstyle{definition}
\newtheorem{theorem}{Theorem}[section]
\newtheorem{remark}[theorem]{Remark}

\newtheorem{lemma}[theorem]{Lemma}
\newtheorem{proposition}[theorem]{Proposition}
\newtheorem{corollary}[theorem]{Corollary}

\newtheorem{definition}[theorem]{Definition}
\newtheorem{example}[theorem]{Example}

\newcommand{\subject}[1]{\begin{flushleft}
\textbf{2010 AMS Subject Classification}: #1\end{flushleft}}
\newcommand{\keyword}[1]{\par\noindent \textbf{Keywords:} #1}

\newcommand{\eval}[2][\right]{\relax
\ifx#1\right\relax \left.\fi#2#1\rvert}

\renewcommand{\sectionmark}[1]{}
\begin{document}
\maketitle
\begin{abstract}
\noindent
Considering a finite Borel measure $ \mu $ on $ \mathbb{R}^d $, a pair of conjugate exponents $ p, q $, and a compatible semi-inner product on $ L^p(\mu) $, we introduce $ (p,q) $-Bessel and $ (p,q) $-frame measures as a generalization of the concepts of Bessel and frame measures. In addition, we define notions of $ q $-Bessel and $ q$-frame in the semi-inner product space $ L^p(\mu) $. Every finite Borel measure $\nu$ is a $(p,q)$-Bessel measure for a finite measure $ \mu $. We construct a large number of examples of finite measures $ \mu $ which admit infinite $ (p,q) $-Bessel measures $ \nu $. We show that if $ \nu $ is a $ (p,q) $-Bessel/frame measure for $ \mu $, then $ \nu $ is $ \sigma $-finite and it is not unique. In fact, by using convolutions of probability measures, one can obtain other $ (p,q) $-Bessel/frame measures for $ \mu $. We present a general way of constructing a $ (p,q) $-Bessel/frame measure for a given measure.
\vspace{.3cm}
\keyword{Fourier frame, Plancherel theorem, spectral measure, frame measure, Bessel measure, semi-inner product.}
\subject{Primary 28A99, 46E30, 42C15.}
\end{abstract}

\section{Introduction}
According to \cite{5}, a Borel measure $ \nu $ on $ \mathbb{R}^d $ is called a dual measure for a given measure $ \mu $ on $ \mathbb{R}^d $ if for every
$ f\in L^2(\mu) $,
\begin{equation}\label{11}
\int_{\mathbb{R}^d} |\widehat{f d\mu}(t)|^2 d\nu(t) \simeq \int_{\mathbb{R}^d} |f(x)|^2 d\mu(x),
\end{equation}
where for a function $ f\in L^1(\mu) $ the Fourier transform is given by
\begin{equation*}
\widehat{f d\mu}(t) = \int_{\mathbb{R}^d} f(x) e^{-2\pi it\cdot x} d\mu(x) \qquad (t\in \mathbb{R}^d).
\end{equation*}
Precisely, the equivalence in Equation (\ref{11}) means that there are positive constants $A$ and $B$
independent of the function $f(x)$ such that
\begin{equation*}
A\int_{\mathbb{R}^d} |f(x)|^2 d\mu(x) \leq \int_{\mathbb{R}^d} |\widehat{f d\mu}(t)|^2 d\nu(t) \leq B\int_{\mathbb{R}^d} |f(x)|^2 d\mu(x).
\end{equation*}
Therefore when $ A=B=1 $, by Plancherel's theorem for Lebesgue measure $ \lambda $ on $ \mathbb{R}^d $, $ \lambda $ is a dual measure to itself. Dual measures are in fact a generalization of the concept of Fourier frames and they are also called frame measures. According to \cite{5}, if $ \mu $ is not an $ F $-spectral measure, then there cannot be any general statement about the existence of frame measures $ \nu $. Nevertheless, the authors showed that if one frame measure exists, then by using convolutions of measures, many frame measures can be obtained, especially a frame measure which is absolutely continuous with respect to Lebesgue measure. Moreover, they presented a general way of constructing Bessel/frame measures for a given measure.

In this paper we generalize the notion of Bessel/frame measure from Hilbert spaces $ L^2(\mu) $, $ L^2(\nu) $ to Banach spaces $ L^p(\mu) $, $ L^q(\nu) $ ($ p, q $ are conjugate exponents) via a compatible semi-inner product defined on $ L^p(\mu) $. Compatible semi-inner products are natural substitutes for inner products on Hilbert spaces. We introduce $ (p,q) $-Bessel and $ (p,q) $-frame measures, and we define notions of $ q $-Bessel and $ q $-frame in the semi-inner product space $ L^p(\mu) $. Then we investigate the existence and some general properties of them. 

The rest of this paper is organized as follows: In section $ 2 $ basic definitions and preliminaries are given. In section $ 3 $ we investigate the existence of $ (p,q) $-Bessel/frame measures. We show that every finite Borel measure $\nu$ is a $(p,q)$-Bessel measure for a finite measure $ \mu $. In addition, we construct a large number of examples of measures which admit infinite discrete $ (p,q) $-Bessel measures, by F-spectral measures and applying the Riesz-Thorin interpolation theorem. In general, for every spectral measure (B-spectral measure, or F-spectral measure respectively) $ \mu $, there exists a discrete measure $\nu = \sum_{\lambda\in\Lambda_\mu}\delta_\lambda $ which is a Plancherel measure (Bessel measure or frame measure respectively) for $ \mu $. Then the Riesz-Thorin interpolation theorem yields that $ \nu $ is also a $ (p,q) $-Bessel measure for $ \mu $, where $ 1 \leq p \leq 2 $ and $ q $ is the conjugate exponent to $ p $. Moreover, this shows that if $ \mu $ is a spectral measure (B-spectral measure, or F-spectral measure), then the set $ \{e_\lambda\}_{\lambda\in\Lambda_\mu}$ forms a $ q $-Bessel for $ L^p(\mu) $. It is known \cite{13, 19} that if a measure $ \mu $ is an F-spectral measure, then it must be of pure type, i.e., $ \mu $ is either discrete, absolutely continuous or singular continuous. Therefore, we consider such measures in constructing the examples. The interested reader can refer to \cite{3, 6, 7, 9, 13, 16, 18, 19, 20, 21, 23, 24} to see examples and properties of spectral measures (B-spectral measures, or F-spectral measures) and related concepts. Besides discrete $ (p,q) $-Bessel measures $\nu = \sum_{\lambda\in\Lambda_\mu}\delta_\lambda $ associated to spectral measures (B-spectral measures, or F-spectral measures) $ \mu $, we prove that there exists an infinite absolutely continuous $ (p,q) $-Bessel measure $ \nu $ for a special finite measure $ \mu $. We show that if $ \nu $ is a $ (p_1,q_1) $-Bessel/frame measure and $ (p_2,q_2) $-Bessel/frame measure for $ \mu $, where $1 \leq p_1, p_2 < \infty$ and $ q_1, q_2  $ are the conjugate exponents to $ p_1, p_2 $, respectively, then $ \nu $ is a $ (p,q) $-Bessel measure for $ \mu $ too, where $ p_1 < p < p_2 $ and $ q $ is the conjugate exponent to $ p $. Consequently, if $ \nu $ is a Bessel/frame measure for $ \mu $, then it is a $ (p,q) $-Bessel measure for $ \mu $ too. In Proposition \ref{3.20} we prove that there exists a measure $ \mu $ which admits tight $ (p,q) $-frame measures and $ (p,q) $-Plancherel measures.\\ 
Section $ 4 $ is devoted to investigating properties of $ (p,q) $-Bessel/frame measures based on the results by Dutkay, Han, and Weber from \cite{5}. 
\section{ Preliminaries}
\begin{definition}
Let $ t\in \mathbb{R}^d $. Denote by $ e_t $ the exponential function
\begin{equation*}
e_t(x) = e^{2\pi it\cdot x}\qquad (x\in \mathbb{R}^d). 
\end{equation*}
\end{definition}
\begin{definition}
Let $ H $ be a Hilbert space. A sequence $ \{f_i \}_{i \in I} $ of elements in $ H $ is called a \emph{Bessel sequence} for $ H $ if there exists a positive constant $ B $ such that for all $ f \in H $,
\begin{equation*}
\sum_{i \in I }|\left<f, f_i\right>|^2 \leq B\| f\|^2.  
\end{equation*}
Here $ B $ is called the \emph{Bessel bound} for the Bessel sequence $ \{ f_i \}_{i\in I} $.
 
The sequence $ \{f_i \}_{i \in I} $ is called a \emph{frame} for $ H $, if there exist constants $ A, B > 0 $  such that for all $ f \in H $,
\begin{equation*}
 A\| f\|^2 \leq \sum_{i \in I }|\left<f, f_i\right>|^2 \leq B\| f\|^2.  
\end{equation*}
In this case, $ A $ and $ B $ are called \emph{frame bounds}.
\end{definition} 
Frames are a natural generalization of orthonormal bases. It is easily seen from the lower bound that a frame is complete in H, so every $ f $ can
be expressed using (infinite) linear combination of the elements $ f_i $ in the frame \cite{2}. 
\begin{definition}
Let $ \mu $ be a compactly supported probability measure on $ \mathbb{R}^d $ and $\Lambda$ be a countable set in $ \mathbb{R}^d $, the set $ E(\Lambda)=\{e_\lambda : \lambda \in \Lambda \} $ is called a \emph{Fourier frame} for $ L^2(\mu) $ if for all $ f\in L^2(\mu) $, 
\begin{equation*}
A\| f\|_{L^2(\mu)}^2 \leq \sum_{\lambda \in \Lambda }|\left<f, e_\lambda \right>_{L^2(\mu)}|^2 \leq B\| f\|_{L^2(\mu)}^2. 
\end{equation*}

When $ E(\Lambda) $ is an orthonormal basis (Bessel sequence, or frame) for $ L^2(\mu) $, we say that $ \mu $ is a \emph{spectral measure (B-spectral measure, or F-spectral measure} respectively) and $ \Lambda $ is called a \emph{spectrum (B-spectrum, or F-spectrum} respectively) for $ \mu $.
\end{definition}
We give the following definition from \cite{5}, assuming that the given measure $ \mu $ is a finite Borel measure on $ \mathbb{R}^d $.
\begin{definition}[\cite{5}]
Let $ \mu $ be a finite Borel measure on $ \mathbb{R}^d $. A Borel measure $ \nu $ is called a \emph{Bessel measure} for $ \mu $, if there exists a positive constant $ B $ such that for every $ f\in L^2(\mu) $, 
\begin{equation*}
 \|\widehat{fd\mu}\|_{L^2(\nu)}^2 \leq B\| f\|_{L^2(\mu)}^2.
\end{equation*}
Here $ B $ is called a \emph{(Bessel) bound} for $ \nu $.
 
The measure $ \nu $ is called a \emph{frame measure} for $ \mu $ if there exist positive constants $ A, B $ such that for every $ f\in L^2(\mu) $,
\begin{equation*}
A\| f\|_{L^2(\mu)}^2 \leq \|\widehat{fd\mu}\|_{L^2(\nu)}^2 \leq B\| f\|_{L^2(\mu)}^2.
\end{equation*}
In this case, $ A $ and $ B $ are called \emph{(frame) bounds} for $ \nu $. The measure $ \nu $ is called a tight frame measure if $ A = B $ and Plancherel measure if $ A = B =1 $ (see also \cite{8}).

The set of all Bessel measures for $ \mu $ with fixed bound $ B $ is denoted by $ \mathcal{B}_B(\mu) $ and the set of all frame measures for $ \mu $ with fixed bounds $ A, B $ is denoted by $ \mathcal{F}_{A,B}(\mu) $.
\end{definition}
\begin{remark} \label{2.5}
A compactly supported probability measure $ \mu $ is an F-spectral measure if and only if there exists a countable set $\Lambda $ in $ \mathbb{R}^d $ such that $ \nu=\sum_{\lambda \in \Lambda}\delta_\lambda $ is a frame measure for $ \mu $.
\end{remark}
\begin{definition}
 A finite set of contraction mappings $ \{\tau_i\}_{i=1}^n $ on a complete metric space is called an \emph{iterated function system (IFS)}. Hutchinson \cite{15} proved that, for the metric space $ \mathbb{R}^d $, there exists a unique compact subset $ X $ of $ \mathbb{R}^d $, which satisfies $ X=\bigcup_{i=1}^n \tau_i(X) $. Moreover, if the IFS is associated with a set of probability weights $ \{\rho_i\}_{i= 1}^n $ (i.e., $ 0 <\rho_i < 1 $,  $ \sum_{i=1}^n \rho_i =1 $), then there exists a unique Borel probability measure $ \mu $ supported on $ X $ such that $ \mu=\sum_{i=1}^n \rho_i (\mu o\tau^{-1}) $. The corresponding $ X $ and $ \mu $ are called the \emph{attractor} and the \emph{invariant measure} of the IFS, respectively. It is well known that the invariant measure is either absolutely continuous or singular continuous with respect to Lebesgue measure.
In an affine IFS each $ \tau_i $ is affine and represented by a matrix. If $R$ is a $ d \times d $ expanding integer matrix (i.e., all eigenvalues $ \lambda $
satisfy $ |\lambda|>1 $), and $\mathcal{A} \subset \mathbb{Z}^d $, with $ \#\mathcal{A} =: N \geq 2 $, then the following set (associated with a set of probability weights) is an affine iterated function system.
\begin{equation*}
\tau_a (x) = R^{-1}(x + a) \quad (x \in \mathbb{R}^d, a \in \mathcal{A}).
\end{equation*}
Since $R$ is expanding, the maps $\tau_a$ are contractions (in an appropriate metric
equivalent to the Euclidean one). In some cases, the invariant measure $ \mu_\mathcal{A} $ is a \emph{fractal measure} (see \cite{3}). For example singular continuous invariant measures supported on Cantor type sets are fractal measures (see \cite{15, 14}). 
\end{definition} 
\begin{definition}[\cite{22}](\emph{Semi-inner product spaces}) \\
Let $ X $ be a vector space over the filed $ F $ of complex (real) numbers. If a function $ [\cdot , \cdot] : X \times X \rightarrow F $ satisfies the following properties:
 \begin{enumerate}
\item[1.] $[x + y, z] = [x, z] + [y, z], \;\;\ \text{for}\ x, y, z \in X;$
\item[2.] $[\lambda x, y] = \lambda[x, y], \;\ \text{for} \ \lambda\in F \ \text{and} \ x, y \in X;$
\item[3.] $[x, x] > 0, \;\ \text{for} \ x \neq 0;$
\item[4.] $|[x, y]|^2 \leq [x, x][y, y],$
\end{enumerate}
then $ [\cdot , \cdot] $ is called a \emph{semi-inner product} and the pair$ (X, [\cdot, \cdot]) $ is called a \emph{semi-inner product space}. It is easy to observe that $\|x\| = [x, x]^\frac{1}{2}$ is a norm on $ X $. So every semi-inner product space is a normed linear space. On the other hand, one can generate a semi-inner product in a normed linear space, in infinitely many different ways. 
\end{definition}
As a matter of fact, semi-inner products provide the possibility of carrying over Hilbert space type arguments to Banach spaces.

Every Banach space has a semi-inner product that is compatible. For example consider the Banach function space $ L^p(X , \mu),\; p \geq 1$, a compatible semi-inner product on this space is defined by (see \cite{12})
\begin{equation*}
[f,g]_{L^p(\mu)} := \dfrac{1}{\| g\|_{L^p(\mu)}^{p-2}}\int_X f(x) |g(x)|^{p-1}\overline{sgn(g(x))} d\mu(x),
\end{equation*} 
for every $ f, g \in L^p(X , \mu) $ with $ \|g\|_{L^p(\mu)} \neq 0 $, and $ [f,g]_{L^p(\mu)} = 0 $ for $ \|g\|_{L^p(\mu)} = 0 $.\\
  
To construct frames in a Hilbert space $ H $ the sequence space $ l^2 $ is required. Similarly, to construct frames in a Banach space $ X $ one needs a Banach space of scaler valued sequences $ X_d $ (in fact a BK-space $ X_d $, see \cite{1} and the references therein). According to Zhang and Zhang \cite{26} frames in Banach spaces can be defined via a compatible semi-inner product in the following way: 
\begin{definition}\label{2.7}
Let $ X $ be a Banach space with a compatible semi-inner product $ [\cdot , \cdot] $ and norm $ \|\cdot\|_X $. Let $ X_d $ be an associated BK-space with norm  $ \|\cdot\|_{X_d} $. A sequence of elements $ \{ f_ i \}_{i\in I} \subseteq X $ is called an \emph{$ X_d $-frame} for $ X $ if $ \{[ f, f_ i ]\}_{i\in I} \in X_d $ for all $ f \in X $, and there exist constants $ A, B > 0 $ such that for every $ f \in X $,
\begin{equation*}
 A\| f\|_X \leq \|\{[f,f_i]\}_{i\in I} \|_{X_d} \leq B\| f\|_X.
\end{equation*}
See also \cite{25}.
\end{definition} 
Based on Definition \ref{2.7}, we present the next definition. We consider the function space $ L^p( \mu)$ and the sequence space $ l^q(I) $ (where $ p>1 $ and $ q $ is the conjugate exponent to $ p $) as the Banach space and the BK- space, respectively.  
\begin{definition}
Suppose that $ 1<p, q<\infty $ and $ \dfrac{1}{p} + \dfrac{1}{q} =1 $. Let $ \mu  $ be a finite Borel measure on $ \mathbb{R}^d $ and let $ [\cdot ,\cdot] $ be the compatible semi-inner product on $ L^p(\mu) $ as defined above. We say that a sequence $ \{f_i\}_{i\in I} $ is a \emph{$ q $-Bessel} for $ L^p(\mu) $ if there exists a constant $ B > 0 $ such that for every $ f\in L^p(\mu) $,
\begin{equation*}
 \sum_{i\in I}|[f,f_i]_{L^p(\mu)}|^q\leq B\| f\|_{L^p(\mu)}^q.
\end{equation*}
We call B a \emph{($ q $-Bessel) bound}. 

We say the sequence $ \{f_i\}_{i\in I} $ is a \emph{$ q $-frame} for $ L^p(\mu) $ if there exist constants $A, B > 0 $ such that for every $ f\in L^p(\mu) $, 
\begin{equation*}
 A\| f\|_{L^p(\mu)}^q\leq\sum_{i\in I}|[f,f_i]_{L^p(\mu)}|^q\leq B\| f\|_{L^p(\mu)}^q.
\end{equation*}
We call $ A,B $ \emph{($ q $-frame) bounds}. We call the sequence $ \{f_i\}_{i\in I} $ a \emph{tight $ q $-frame} if $ A = B $ and \emph{Parseval $ q $-frame} if $ A = B = 1 $. 
\end{definition} 
 We extend the notions of Bessel and frame measures as follows.
\begin{definition}
Suppose that $ 1\leq p<\infty, 1<q\leq\infty $ and $ \dfrac{1}{p} + \dfrac{1}{q} =1 $. Let $ \mu  $ be a finite Borel measure on $ \mathbb{R}^d $, and let $ [\cdot ,\cdot] $ be the compatible semi-inner product on $ L^p(\mu) $ as defined above. We say that a Borel measure $ \nu $ is a \emph{$ (p,q) $-Bessel measure} for $ \mu $, if there exists a constant $ B > 0 $ such that for every $ f\in L^p(\mu) $, 
\begin{equation*}
  \int_{\mathbb{R}^d} |[f,e_t]_{L^p(\mu)}|^q d\nu(t)\leq B\| f\|_{L^p(\mu)}^q \quad  (p \neq 1, q\neq \infty) 
\end{equation*}
and  
\begin{equation*}
 \|\widehat{fd\mu}\|_\infty \leq B\| f\|_{L^1(\mu)} \quad  (p=1, q=\infty).  
\end{equation*}
We call $ B $ a (\emph{$ (p,q) $-Bessel}) \emph{bound} for $ \nu $. 
 
We say the Borel measure $ \nu $ is a \emph{$ (p,q) $-frame measure} for $ \mu $, if there exist constants $ A,B > 0 $ such that for every $ f\in L^p(\mu) $, 
\begin{equation*}
 A\| f\|_{L^p(\mu)}^q\leq \int_{\mathbb{R}^d} |[f,e_t]_{L^p(\mu)}|^q d\nu(t)\leq B\| f\|_{L^p(\mu)}^q \quad  (p \neq 1, q \neq \infty) 
\end{equation*}
and 
\begin{equation*}
 A\| f\|_{L^1(\mu)}\leq  \|\widehat{fd\mu}\|_\infty \leq B\| f\|_{L^1(\mu)} \quad  (p=1, q=\infty).
\end{equation*}
We call $ A,B $ (\emph{$ (p,q) $-frame}) \emph{bounds} for $ \nu $. If $ A = B $, we call the measure $ \nu $ a \emph{tight $ (p,q) $-frame measure} and  if $ A = B =1 $, we call it a \emph{$ (p,q) $-Plancherel measure}.

We denote the set of all $(p,q)$-Bessel measures for $ \mu $ with fixed bound $ B $ by $ \mathcal{B}_B(\mu)_{p,q} $ and  the set of all $(p,q)$-frame measures for $ \mu $ with fixed bounds $A, B $ by $ \mathcal{F}_{A,B}(\mu)_{p,q} $.
\end{definition}
\begin{remark} \label{2.11}
Since $[f,e_t]_{L^p(\mu)}= \int_{R^d} f(x) e^{-2\pi it\cdot x} d\mu(x) = \widehat{f d\mu}(t)$ for any $ f\in L^p(\mu) $ and $ t \in \mathbb{R}^d$, we can also write $ \widehat{f d\mu}(t) $ instead of $ [f,e_t]_{L^p(\mu)} $.
 If there exists a $ (p,q) $-Bessel/frame measure $ \nu $ for $ \mu $, then the function $ T_\nu : L^p(\mu) \rightarrow L^q(\nu) $ defined by $ T_\nu f=\widehat{fd\mu} $ is linear and bounded. For $ p=1 $, $ q= \infty $, every $\sigma$-finite measure $ \nu $ on $ \mathbb{R}^d $ is a $(1, \infty)$-Bessel measure for $ \mu $, since we always have $ \| \widehat{fd\mu} \|_\infty \leq \| f \|_{L^1(\mu)} $. More precisely, $ \nu \in \mathcal{B}_1(\mu)_{(1, \infty)} $.
\end{remark}
\begin{theorem}[\cite{10}]
\emph{(Riesz-Thorin interpolation theorem)} Let $1 \leq p_0, p_1, q_0, q_1 \leq \infty$, where
$p_0 \neq p_1$ and $q_0 \neq q_1$, and let $T$ be a linear operator.  Suppose that for some measure spaces $(Y, \nu)$, $(X, \mu)$, $ T: L^{p_0}(X, \mu)\rightarrow L^{q_0}(Y, \nu) $ is bounded with norm $ C_0 $, and $ T: L^{p_1}(X, \mu)\rightarrow L^{q_1}(Y, \nu) $ is bounded with norm $ C_1 $. Then for all $ \theta \in (0, 1) $ and $ p, q $ defined by $ \dfrac{1}{p} = \dfrac{(1 - \theta)}{p_0} + \dfrac{\theta}{p_1}; \ \dfrac{1}{q} =\dfrac{1- \theta}{q_0} + \dfrac{\theta}{q_1}$, there exists a constant $ C $ such that $ C\leq C_0^{(1-\theta)} C_1^{\theta}$ and $ T: L^{p}(X, \mu)\rightarrow L^{q}(Y, \nu) $ is bounded with norm $ C $.
\end{theorem}
All measures we consider in this paper, are Borel measures on $ \mathbb{R}^d $.  
\section{Existence and Examples}
In this section we investigate the existence of $ (p,q) $-Bessel and $ (p,q) $-frame measures and also the existence of $ q $-Bessel and $ q $-frame sequences. In addition, we construct examples of measures which admit $ (p,q) $-Bessel measures. 
\begin{proposition}\label{3.1}
Suppose that $ 1<p, q<\infty $ and $ \dfrac{1}{p} + \dfrac{1}{q} =1 $. Let $ \mu $ be a finite Borel measure. Then every finite Borel measure $\nu$ is a $(p,q)$-Bessel measure for $ \mu $.
\end{proposition} 
\begin{proof}
 Take $f \in L^p (\mu)$ and $ t \in \mathbb{R}^d$. Then by applying Holder's inequality  
\begin{equation*}
|[f, e_t]_{L^p(\mu)}| \leq \int_{\mathbb{R}^d} |f(x) e^{-2\pi it\cdot x}| d\mu(x)\leq (\mu(\mathbb{R}^d))^{\frac{1}{q}}\| f\|_{L^p(\mu)} .
\end{equation*}
Thus,
\begin{equation*}
 \int_{\mathbb{R}^d} |[f, e_t]_{L^p(\mu)}|^q d\nu(t) \leq \mu(\mathbb{R}^d) \nu(\mathbb{R}^d)\| f\|^q_{L^p(\mu)}.
 \end{equation*}
Therefore $\nu \in \mathcal{B}_{\mu(\mathbb{R}^d) \nu(\mathbb{R}^d)}(\mu)_{(p, q)} $. For $ p=1 $, $ q= \infty $, as we mentioned in Remark \ref{2.11} $ \nu \in \mathcal{B}_1(\mu)_{(1, \infty)} $.
\end{proof}
\begin{proposition}\label{3.2}
Suppose that $ 1<p, q<\infty $ and $ \dfrac{1}{p} + \dfrac{1}{q} =1 $. Let $ \Lambda \subset \mathbb{R}^d $, $ \#\Lambda < \infty $ and $ \mu $ be a finite Borel measure. Then the finite sequence $ \{e_\lambda\}_{\lambda\in \Lambda} $ is a $ q $-Bessel for $ L^p(\mu) $.
\end{proposition}
\begin{proof}
Consider the finite discrete measure $\nu = \sum_{\lambda\in\Lambda}\delta_\lambda $. Since
\begin{equation*}
 \sum_{\lambda\in\Lambda}|[f, e_\lambda]_{L^p(\mu)}|^q =\int_{\mathbb{R}^d} |[f, e_t]_{L^p(\mu)}|^q d\nu(t),
 \end{equation*}
 then the assertion follows from Proposition \ref{3.1}.
\end{proof}
\begin{remark}
Proposition \ref{3.1} shows that the Bessel bound may change for different measures $ \nu $. So if we consider Borel probability measures $ \nu $, then we have a fixed Bessel bound $ \mu(\mathbb{R}^d) $ for all $ \nu $. Moreover, this Bessel bound does not depend on $p, q$, i.e., for every probability measure $\nu$ we have $\nu \in \mathcal{B}_{\mu(\mathbb{R}^d)}(\mu)_{(p, q)} $, where $ 1 < p< \infty $ and $ q $ is the conjugate exponent to $ p $. In addition, we obtain from  Proposition \ref{3.1} that for all conjugate exponents $ p, q > 1 $ the set $ \mathcal{B}_{\mu(\mathbb{R}^d)}(\mu)_{p,q} $ is infinite, since there are infinitely many probability measures $ \nu $ (such as every measure $ \nu=\frac{1}{\lambda(S)}\chi_S d\lambda $ where $ S \subset \mathbb{R}^d$ with the finite Lebesgue measure $ \lambda(S) $, every finite discrete measure $ \nu = \frac{1}{n}\sum_{a=1}^n \delta_a $ where $ \delta_a $ denotes the Dirac measure at the point $ a $, every invariant measure obtained from an iterated function system, and others).
\end{remark}
\begin{proposition} \label{3.3}
Suppose that $ 1<p, q<\infty $ and $ \dfrac{1}{p} + \dfrac{1}{q} =1 $. Let $ \nu $ be a finite Borel measure. Then $ \nu $ is a $(p, q)$-Bessel measure for every finite Borel measure $ \mu $. In addition, $\nu \in \mathcal{B}_{\nu(\mathbb{R}^d)}(\mu)_{(p, q)} $ for all probability measures $ \mu $.
\end{proposition}
\begin{proof}
See the proof of Proposition \ref{3.1}.
\end{proof}
\begin{corollary}
Suppose that $ 1<p, q<\infty $ and $ \dfrac{1}{p} + \dfrac{1}{q} =1 $. A finite Borel measure $ \nu $ is  a $ (p,q) $-Bessel measure for a finite Borel measure $ \mu $, if and only if $ \mu $ is a $ (p,q) $-Bessel measure for $ \nu $. In particular, every finite Borel measure $ \mu $ is a $ (p,q) $-Bessel measure to itself. 
\end{corollary}
\begin{proof}
The statements are direct consequences of Propositions \ref{3.1} and \ref{3.3}.
\end{proof} 
\begin{lemma} \label{3.6}
Suppose that $ 1<p, q<\infty $ and $ \dfrac{1}{p} + \dfrac{1}{q} =1 $. Let $ \mu $ be a finite Borel measure. Then the following assertions hold.

(i) If there exists a countable set $\Lambda $ in $ \mathbb{R}^d $ such that $\{e_\lambda\}_{\lambda \in \Lambda} $ is a $ q $-frame for $L^p (\mu) $, then $\nu = \sum_{\lambda\in\Lambda}\delta_\lambda $ is a $ (p,q) $-frame measure for $\mu$. 

(ii) If $ \nu $ is purely atomic, i.e. $\nu = \sum_{\lambda\in\Lambda} d_\lambda \delta_\lambda $, and a $(p,q)$-frame measure for the probability measure $ \mu $, then $\lbrace\sqrt [q]{d_\lambda}\; e_\lambda \rbrace_{\lambda\in\Lambda}  $ is a q-frame for $L^p (\mu) $.
\end{lemma}
\begin{proof}
(i) Let $\nu = \sum_{\lambda\in\Lambda}\delta_\lambda $. Then for all $f \in L^p (\mu)$,  
\begin{align*}
\sum_{\lambda\in\Lambda} |[f,e_\lambda]_{L^p(\mu)}|^q =\int_{\mathbb{R}^d} |[f,e_t]_{L^p(\mu)}|^q d\nu(t).
\end{align*}

(ii) Since for all $f \in L^p (\mu)$, 
\begin{equation*}
\int_{\mathbb{R}^d} |[f,e_t]_{L^p(\mu)}|^q d\nu(t) = \sum_{\lambda\in\Lambda}d_\lambda|[f,e_\lambda]_{L^p(\mu)}|^q = \sum_{\lambda\in\Lambda} |[f,\sqrt [q]{d_\lambda} e_\lambda]|^q.
\end{equation*}
\end{proof}
\begin{example} \label{ex 4}
Suppose that $1\leq  p \leq 2 $ and $ q $ is the conjugate exponent to $ p $. If $ f \in  L^p ([0 , 1]^d) $, then from the Hausdorff-Young inequality we have $ \hat{f} \in l^q(\mathbb{Z}^d) $ and $ \| \hat{f} \|_q \leq \| f\|_p $. Therefore the measure $ \nu=\sum_{t \in \mathbb{Z}^d} \delta_t $ is a $ (p,q) $-Bessel measure for $ \mu= \chi_{\{[0 , 1]^d\}}dx $. Besides, $\{e_t\}_{t \in \mathbb{Z}^d}$ is a $ q $-Bessel for $ L^p (\mu) $, since $ \sum_{t\in \mathbb{Z}^d} |[f , e_t]_{ L^p (\mu)}|^q \leq \| f\|_{p}^q $, where $1<  p \leq 2 $ and $ q $ is the conjugate exponent to $ p $.
\end{example}
\begin{proposition} 
Let $1< p \leq 2 $ and $ q $ is the conjugate exponent to $ p $. Let $ 0<a\leq \phi (x)\leq b <\infty $ on $ [0, 1]^d $ and $ \phi_t (x):=\phi (x) $ for all $ t \in \mathbb{Z}^d $. Then $\{\phi_te_t\}_{t \in \mathbb{Z}^d}$ is a $ q $-Bessel for  $ L^p ([0 , 1]^d) $.
\end{proposition}
\begin{proof}
 Take $ f\in L^p \left([0 , 1]^d\right) $. We have $ \dfrac{1}{\| \phi \|_p^{p-2}}\phi^{p-1} f\in L^p ([0 , 1]^d) $, since 
\begin{equation*}
\int_{[0 , 1]^d}  |f(x)|^p \left| \dfrac{ \phi^{p-1}(x)}{\| \phi \|_p^{p-2}} \right|^p d(x) \leq \dfrac{b^{(p-1)p}}{a^{(p-2)p}} \int_{[0 , 1]^d} |f(x)|^p d(x)<\infty.
\end{equation*}
Hence by Example \ref{ex 4},
\begin{align*}
\sum_{t\in \mathbb{Z}^d} \left|[f , \phi_t e_t]_{L^p ([0 , 1]^d)} \right|^q &= \sum_{t\in \mathbb{Z}^d} \left| \dfrac{1}{\| \phi \|_p^{p-2}} \int_{[0 , 1]^d} f(x) \left|\phi(x)e_t(x) \right|^{p-1}e_{-t}(x)dx \right|^q \\ 
&\leq \left| \int_{[0 , 1]^d}  |f(x)|^p \left| \dfrac{ \phi^{p-1}(x)}{\| \phi \|_p^{p-2}} \right|^p dx\right|^{q/p}\leq \dfrac{b^p}{a^{p-q}} \| f\|_p^q.
\end{align*}
\end{proof}
\begin{corollary}Suppose that $ 1<p, q<\infty $ and $ \dfrac{1}{p} + \dfrac{1}{q} =1 $. Let $ \mu $ be a probability measure. Let $ 0<a\leq \phi (x)\leq b <\infty $ on $ supp\mu $ and $ \phi_i (x):=\phi (x) $ for all $ i\in I $. If $\{f_i\}_{i\in I}$ is a $ q $-frame for $ L^p (\mu) $, then $ \{\phi_if_i\}_{i \in I}$ is also a $ q $-frame for $ L^p (\mu) $ and for every $ f\in L^p(\mu) $,
\begin{equation*}
\dfrac{a^p}{b^{p-q}} A\parallel f\parallel_{L^p(\mu)}^q\leq\|\{[f,\phi_if_i]_{L^p(\mu)}\}_{i\in I} \|^q\leq \dfrac{b^p}{a^{p-q}} B\| f\|_{L^p(\mu)}^q.
\end{equation*}
\end{corollary} 
\begin{remark} Example \ref{ex 4} cannot be extended to the case $ p>2 $, since there exist continuous functions $ f $ such that $ \sum_{n\in \mathbb{Z}} |[f , e_n]_{L^p(\mu)}|^{2-\epsilon} =\infty $ for all $ \epsilon>0 $. An example of such a function is\\
 $ f(x) = \sum_{n=2}^\infty \dfrac{e^{i n\log n}}{n^{1/2}(\log n)^2}e^{inx} $ (see \cite{17}). Therefore $\{e_n\}_{n \in \mathbb{Z}}$ is not a $ q $-Bessel for $ L^p ([0 , 1]) $ and also $ \nu=\sum_{n \in \mathbb{Z}} \delta_n $ is not a $ (p,q) $-Bessel measure for $ \mu= \chi_{[0 , 1] }dx $ where $ p>2 $.
 \end{remark}
\begin{proposition}\label{2.21}
Suppose that $ 1<p, q<\infty $ and $ \dfrac{1}{p} + \dfrac{1}{q} =1 $. Let $ \mu $ be a compactly supported Borel probability measure. Consider two subsets of $ \mathbb{R}^d $, $ \Lambda= \{ \lambda_n \; : n \in \mathbb{N} \} $ and $ \Omega= \{ \omega_n \; : n \in \mathbb{N}\} $ with the property that there exists a positive constant $ C $ such that $ |\lambda_n - \omega_n| \leq C $ for $ n\in \mathbb{N}$. 

(i) If $\{e_{\lambda_n}\}_{n \in \mathbb{N}}$ is a $ q $-Bessel for $ L^p (\mu) $, then $\{e_{\omega_n}\}_{n \in \mathbb{N}} $ is a $ q $-Bessel too.

(ii) If $\{e_{\lambda_n}\}_{n \in \mathbb{N}}$ is a $ q $-frame for $ L^p (\mu) $, then there exists a $ \delta >0 $ such that if $ C\leq \delta $ then $\{e_{\omega_n}\}_{n \in \mathbb{N}} $ is a $ q $-frame too (see \cite{3}).
\begin{proof}
We need only consider the case, when all $ \omega_n = \left( (\omega_n)_1, \ldots, (\omega_n)_d \right) $ differ from $ \lambda_n = \left( (\lambda_n)_1, \ldots, (\lambda_n)_d \right)  $ just on the first component, then the assertion follows by induction on the number of components.\\
Let supp$ \mu \subseteq [-M , M]^d $ for some $ M>0 $. Let $ f\in L^p (\mu) $ and $ x\in \mathbb{R}^d $. The function $ \widehat{fd\mu} $ is analytic in each variable $ t_1,\ldots, t_d $. Moreover,
\begin{equation*}
\dfrac{\partial^k \widehat{fd\mu}}{\partial t_1 ^k} (t) =\int f(x) (-2\pi i x_1)^k e^{-2\pi i t\cdot x} d\mu (x)= \left[(-2\pi i x_1)^k f, e_t\right]_{L^p (\mu)}, \quad ( t \in \mathbb{R}^d ).
\end{equation*}
Writing the Taylor expansion at $(\lambda_n)_1$ in the first variable and using Holder's inequality, for all $ n\in \mathbb{N} $,      
\begin{align*}
\lvert \widehat{fd\mu}(\omega_n) -\widehat{fd\mu}(\lambda_n) \rvert^q &= \left| \sum_{k=1} ^\infty \dfrac{\dfrac{\partial^k \widehat{fd\mu}}{\partial t_1 ^k}(\lambda_n)}{k!} ( (\omega_n)_1 - (\lambda_n)_1 )^k \right|^q \\
&\leq \sum_{k=1}^\infty \dfrac{\left|\dfrac{\partial^k \widehat{fd\mu}}{\partial t_1 ^k}(\lambda_n)\right|^q}{k!}\cdot \left(\sum_{k=1}^\infty \dfrac{|(\omega_n)_1 - (\lambda_n)_1 |^{pk}}{k!}\right)^{q/p}\\
&\leq \sum_{k=1}^\infty \dfrac{\left|\dfrac{\partial^k \widehat{fd\mu}}{\partial t_1 ^k}(\lambda_n)\right|^q}{k!}\cdot \left( \sum_{k=1}^\infty \dfrac{C^{pk}}{k!}\right)^{q-1}=\sum_{k=1}^\infty \dfrac{\left|\dfrac{\partial^k \widehat{fd\mu}}{\partial t_1 ^k}(\lambda_n)\right|^q}{k!}\cdot \left(e^{C^p}-1\right)^{q-1}. 
\end{align*}
Considering the $ q $-Bessel $\{e_{\lambda_n}\}_{n \in \mathbb{N}}$ with a bound $ B $, we obtain
\begin{equation*}
\sum_{n\in \mathbb{N}}\left|\dfrac{\partial^k \widehat{fd\mu}}{\partial t_1 ^k}(\lambda_n)\right|^q= \sum_{n\in \mathbb{N}} \left| \left[ (-2\pi i x_1)^k f, e_{\lambda_n} \right]_{L^p (\mu)}\right|^q\leq B\|(-2\pi i x_1)^k f\|_{L^p(\mu)}^q\leq B (2\pi M)^{qk} \| f\|_{L^p(\mu)}^q.
\end{equation*}
Then
\begin{align*}
\sum_{n\in \mathbb{N}}\lvert \widehat{fd\mu}(\omega_n) -\widehat{fd\mu}(\lambda_n) \rvert^q &\leq B \left(e^{C^p}-1\right)^{q-1} \| f\|_{L^p(\mu)}^q \sum_{k=1}^\infty \dfrac{(2\pi M)^{qk}}{k!} \\
&= B \left(e^{C^p}-1\right)^{q-1} \left(e^{{(2\pi M)}^q}-1\right) \| f\|_{L^p(\mu)}^q.  
\end{align*}
Hence by Minkowski's inequality, 
\begin{align*}
\left(\sum_{n\in \mathbb{N}}| \widehat{fd\mu}(\omega_n) |^q \right)^{1/q} &\leq \left(\sum_{n\in \mathbb{N}}| \widehat{fd\mu}(\lambda_n) |^q \right)^{1/q} + \left(\sum_{n\in \mathbb{N}}\lvert \widehat{fd\mu}(\omega_n) -\widehat{fd\mu}(\lambda_n) \rvert^q\right)^{1/q}\\
&\leq \left( B^{1/q} + \left(B\left( e^{C^p}-1\right)^{q-1} \left(e^{{(2\pi M)}^q}-1\right)\right)^{1/q} \right) \| f\|_{L^p(\mu)}, 
\end{align*}
and this implies that $\{e_{\omega_n}\}_{n \in \mathbb{N}} $ is a $ q $-Bessel for $ L^p (\mu) $.\\
To show that $\{e_{\omega_n}\}_{n \in \mathbb{N}} $ is also a $ q $-frame for $ L^p (\mu) $, let $ A $ be a lower bound for $\{e_{\lambda_n}\}_{n \in \mathbb{N}}$. Take $ \delta > 0 $ small enough such that for $ 0 < C \leq \delta $,
\begin{equation*}
A^{1/q} -  \left(B\left( e^{C^p}-1\right)^{q-1} \left(e^{{(2\pi M)}^q}-1\right)\right)^{1/q} > 0.
\end{equation*}
Then, by Minkowski's inequality,
\begin{align*}
\left(\sum_{n\in \mathbb{N}}| \widehat{fd\mu}(\omega_n) |^q \right)^{1/q} &\geq \left(\sum_{n\in \mathbb{N}}| \widehat{fd\mu}(\lambda_n) |^q \right)^{1/q} - \left(\sum_{n\in \mathbb{N}}\lvert \widehat{fd\mu}(\omega_n) -\widehat{fd\mu}(\lambda_n) \rvert^q\right)^{1/q}\\
&\geq \left( A^{1/q} - \left( B\left( e^{C^p}-1\right)^{q-1} \left(e^{{(2\pi M)}^q}-1\right)\right)^{1/q} \right) \| f\|_{L^p(\mu)}. 
\end{align*}
Thus the assertion follows.
\end{proof}
\begin{proposition} \label{3.9}
Suppose that  $1 \leq p_0, p_1 < \infty$ and $ q_0, q_1  $ are the conjugate exponents to $ p_0, p_1 $ respectively. If $ \nu $ is a $ (p_0,q_0) $-Bessel measure and a $ (p_1,q_1) $-Bessel measure for $ \mu $, then $ \nu $ is also a $ (p,q) $-Bessel measure for $ \mu $, where $ p_0 < p < p_1 $ and $ q $ is the conjugate exponent to $ p $.
\end{proposition}
\begin{proof}
If $ \nu $ is a $ (p_0,q_0) $-Bessel measure for $ \mu $ with bound $ C $ and also a $ (p_1,q_1) $-Bessel measure with bound $ D $, we have 
\begin{equation*}
\forall f \in  L^{p_0}( \mu) \;\;\;   \| \widehat{fd\mu}\|^{q_0}_{ L^{q_0}(\nu)} \leq C  \| f \|_{ L^{p_0}(\mu)}^{q_0},  
\end{equation*} 
and
\begin{equation*}
\forall f \in  L^{p_1}( \mu) \;\;\;    \|\widehat{fd\mu}\|^{q_1}_{ L^{q_1}(\nu)} \leq D \| f \|_{ L^{p_1}(\mu)}^{q_1} .
\end{equation*}

Now if $ \dfrac{1}{p} = \dfrac{(1 - \theta)}{p_0} + \dfrac{\theta}{p_1}; \ \dfrac{1}{q} =\dfrac{1- \theta}{q_0} + \dfrac{\theta}{q_1}$, where $ 0 < \theta < 1 $ (i.e., $ p_0 < p < p_1 $ and $ \dfrac{1}{p} + \dfrac{1}{q} = 1 $), then the Riesz-Thorin interpolation theorem yields
\begin{equation*}
\forall f \in  L^{p}( \mu) \qquad \| \widehat{fd\mu}\|_{ L^q (\nu)}^q \leq B^q \| f \|_{ L^p (\mu)}^q. 
\end{equation*}
where $ B\leq C^{\frac{1}{q_0}(1-\theta)} D^{\frac{1}{q_1}\theta}$ (Considering the fact that if $ p_0 =1 $ and $ q_0=\infty $, then $ C^{\frac{1}{q_0}} $ changes to $ C $, and if $ p_1 =1 $ and $ q_1=\infty $, then $ D^{\frac{1}{q_1}} $ changes to $ D $). 
Hence $ \nu $ is a $ (p,q) $-Bessel measure for $ \mu $, where $ p_0 < p < p_1 $ and $ q $ is the conjugate exponent to $ p $. 
\end{proof}
\begin{corollary}
If $ \nu $ is a Bessel/frame measure for $ \mu $, then $ \nu $ is also a $ (p,q) $-Bessel measure for  $ \mu $, where $ 1\leq p \leq 2 $ and $ q $ is the conjugate exponent to $ p $.
\end{corollary}
\begin{proof} 
Let $ p_0=1, q_0=\infty, p_1=2, q_1=2 $ in the assumption of Proposition \ref{3.9}, then the conclusion follows.
\end{proof}
\begin{proposition} \label{3.13}
 If $ \nu \in \mathcal{F}_{A,B}(\mu) $, then for any constant $\alpha > 0$, $ \nu $ is a frame measure for $ \alpha\mu $. More precisely $ \nu \in \mathcal{F}_{\alpha A,\alpha B}(\alpha\mu) $.
 \end{proposition}
 \begin{proof} Since $ \nu \in \mathcal{F}_{A,B}(\mu) $ for all  $ f\in L^2(\mu) $,
 \begin{equation*} 
A\| f\|_{L^2(\mu)}^2 \leq \|\widehat{fd\mu}\|_{L^2(\nu)}^2 \leq B\| f\|_{L^2(\mu)}^2,  
\end{equation*}
and we have
\begin{equation*} 
\|\widehat{fd\alpha\mu}\|_{L^2(\nu)}^2 = \int_{\mathbb{R}^d} \left|\int_{\mathbb{R}^d} f(x) e_{-t}(x) d\alpha \mu(x)\right|^2 d\nu(t)=\|\widehat{\alpha fd\mu}\|_{L^2(\nu)}^2.
\end{equation*} 
Since $  \alpha f\in L^2(\mu) $, 
\begin{equation*} 
A\| \alpha f\|_{L^2(\mu)}^2 \leq \|\widehat{\alpha fd\mu}\|_{L^2(\nu)}^2 \leq B\| \alpha f\|_{L^2(\mu)}^2\quad \text{for all }  f\in L^2(\mu).  
\end{equation*} 
Therefore, 
\begin{equation*}
\alpha A\| f\|_{L^2(\alpha \mu)}^2 \leq \|\widehat{fd\alpha \mu}\|_{L^2(\nu)}^2 \leq \alpha B\| f\|_{L^2(\alpha \mu)}^2\quad \text{for all }   f\in L^2(\alpha \mu).  
\end{equation*}
Hence
 $ \nu \in \mathcal{F}_{\alpha A,\alpha B}(\alpha\mu) $.
\end{proof}
\begin{theorem}[\cite{23}] \label{2.17}
There exist positive constants $ c, C $ such that for every set $ S \subset \mathbb{R}^d $ of finite measure, there is a discrete set $ \Lambda \subset \mathbb{R}^d $ such that $ E(\Lambda) $ is a frame for $ L^2(S) $ with frame bounds $ c|S| $ and $ C|S| $, where $|S|$ denotes the measure of $S$.
\end{theorem}
\begin{theorem} \label{2.19}
Let $ S $ be a subset (not necessarily bounded) of $ \mathbb{R}^d $ with finite Lebesgue measure $|S|$. Then the probability measure $ \mu=\frac{1}{|S|}\chi_S dx $ has an infinite discrete $(p,q)$-Bessel measure $ \nu $, where $1\leq p \leq 2 $ and $ q $ is the conjugate exponent to $ p $.
\end{theorem}
\begin{proof}
By Theorem \ref{2.17}, There are positive constants $ c, C $ such that for every set $ S \subset \mathbb{R}^d $ 
of finite Lebesgue measure $|S|$, there is a discrete set $ \Lambda \subset \mathbb{R}^d $ such that $ E(\Lambda) $ is a frame for $ L^2(S) $
with frame bounds $ c|S|$ and $ C|S|$. Then by considering the upper bound of the frame, we have
\begin{equation*}
 \sum_{\lambda \in \Lambda }|\left<f,e_\lambda\right>|^2 \leq C|S|\| f\|_{L^2(S)}^2 \quad \text{for all }  f\in L^2(S). 
\end{equation*}
Let $ \mu=\frac{1}{|S|}\chi_S dx $, and then by Proposition \ref{3.13}, 
\begin{equation*}
 \sum_{\lambda \in \Lambda }|\left<f,e_\lambda\right>|^2 \leq C\| f\|_{L^2(\mu)}^2 \quad \text{for all }  f\in L^2(\mu).
\end{equation*}
In addition, $ \|\{[f,e_\lambda]_{L^1(\mu)}\}_{\lambda\in \Lambda} \|_\infty \leq \| f \|_{L^1(\mu)} $, for every $ f $ in $ L^1(\mu) $. Now if $ \dfrac{1}{p} = 1 -\dfrac{\theta}{2} ; \ \dfrac{1}{q} =\dfrac{\theta}{2}$, for $ 0 <\theta < 1 $ (i.e., $1 < p < 2 $ and $ q $ is the conjugate exponent to $ p $), then the Riesz-Thorin interpolation theorem yields
\begin{equation*}
 \sum_{\lambda \in \Lambda }|[f,e_\lambda]_{L^p(\mu)}|^q \leq \mathcal{C}^q\| f\|_{L^p(\mu)}^q \quad \text{for all }  f\in L^p(\mu), 
\end{equation*}
where $ \mathcal{C} \leq C^{\frac{1}{2}\theta} $. Therefore, $ \nu=\sum_{\lambda \in \Lambda} \delta_\lambda $ is a $ (p,q) $-Bessel measure for $ \mu=\frac{1}{|S|}\chi_S dx $, and we have $ \nu \in \mathcal{B}_{\mathcal{C}^q}(\mu)_{(p, q)} $, where $1 < p < 2 $ and $ q $ is the conjugate exponent to $ p $. Moreover, $ \nu \in \mathcal{B}_{C}(\mu)_{(2, 2)} $ and $ \nu \in \mathcal{B}_{1}(\mu)_{(1, \infty)} $. On the other hand for every $1 < p < 2 $ and $ q $ (the conjugate exponent to $ p $), $ \{e_\lambda\}_{\lambda \in \Lambda} $ is a $ q $-Bessel for $ L^p(\mu) $, with bound $ \mathcal{C}^q $.
\end{proof}
\end{proposition}
If $ S \subset  \mathbb{R}^d $ is a compact set with positive Lebesgue measure, then by Theorem \ref{2.17}, we always have the measure $ \mu=\frac{1}{|S|}\chi_S dx $ is an F-spectral measure, but whether it is a spectral measure, is related to Fuglede's conjecture \cite{11}. In the following example we consider a spectral measure of this type.  
 \begin{example}\label{2.20}
Let $ \mu=\chi_{\left\{[0 , 1] \cup[2, 3]\right\}}dx $. The set of exponential functions $\{e_\lambda \; : \lambda \in \Lambda :=\mathbb{Z} \cup \mathbb{Z} +\frac{1}{4}\}$ is an orthogonal basis for $L^2(\mu)$ (see \cite{9}). We consider the probability measure $ \mu'=\dfrac{1}{2}\chi_{\left\{[0 , 1] \cup[2, 3]\right\}}dx $. Then for every $ f $ in $ L^2(\mu') $ we have $\sum_{\lambda \in \Lambda }|\left<f,e_\lambda\right>_{L^2(\mu')}|^2 = \| f\|_{L^2(\mu')}^2 $. In addition, for every $ f \in L^1(\mu') $, we have $ \|\{[f,e_\lambda]_{L^1(\mu')}\}_{\lambda \in \Lambda } \|_\infty \leq \| f \|_{L^1(\mu')} $. 
Now by applying the Riesz-Thorin interpolation theorem
 $ \sum_{\lambda \in \Lambda }|[f,e_\lambda]_{L^2(\mu')}|^q \leq \| f\|_{L^p(\mu')}^q$, for all $ f\in L^p(\mu') $, where $1\leq p \leq 2 $ and $ q $ is the conjugate exponent to $ p $. Hence, $ \nu=\sum_{\lambda \in \Lambda} \delta_\lambda $ is a $ (p,q) $-Bessel measure for $ \mu' $, especially $ \nu \in \mathcal{B}_{1}(\mu')_{(p, q)} $, where $1\leq p \leq 2 $ and $ q $ is the conjugate exponent to $ p $. Besides, $ \{e_\lambda\}_{\lambda \in \Lambda} $ is a $ q $-Bessel for $ L^p(\mu') $ with bound $ 1 $, where $1 < p \leq 2 $ and $ q $ is the conjugate exponent to $ p $.
\end{example} 
\begin{proposition}[\cite{20}] \label{3.18}
Let $ \mu(x) = \phi(x) dx $ be a compactly supported absolutely continuous probability measure. Then $ \mu $ is an F-spectral measure if and only if the density function $ \phi(x) $ is bounded above and below almost everywhere on the support (see also \cite{8}). 
\end{proposition}

\begin{corollary}
If the density function of a compactly supported absolutely continuous probability measure $ \mu $ is essentially bounded above and below on the support, then the following assertions hold.

(i) There exists an infinite  $ (p,q) $-Bessel measure $ \nu=\sum_{\lambda \in \Lambda_\mu} \delta_\lambda $ for $ \mu $, where $1\leq p \leq 2 $ and $ q $ is the conjugate exponent to $ p $. Moreover, when $ \mu $ is a spectral measure, we have $ \nu \in \mathcal{B}_1(\mu)_{p, q}$, where $1\leq p \leq 2 $ and $ q $ is the conjugate exponent to $ p $.

(ii) There exists a $ q $-Bessel $ \{e_\lambda\}_{\lambda \in \Lambda_{\mu}} $ for $ L^p(\mu) $, where $1< p \leq 2 $ and $ q $ is the conjugate exponent to $ p $. In addition, when $ \mu $ is a spectral measure,  $ \{e_\lambda\}_{\lambda \in \Lambda_{\mu}} $ is a $ q $-Bessel for $ L^p(\mu) $ with bound $ 1 $, where $1 < p \leq 2 $ and $ q $ is the conjugate exponent to $ p $.
\end{corollary}
\begin{proof}
The conclusion follows from Proposition \ref{3.18} and the Riesz-Thorin interpolation theorem (see the proof of Theorem \ref{2.19} and also, see Example \ref{2.20}).
\end{proof}
By Proposition \ref{3.3}, if $ 1<p, q<\infty $ and $ \dfrac{1}{p} + \dfrac{1}{q} =1 $, then a fixed finite Borel measure $ \nu $ is a $(p, q)$- Bessel measure for every finite measure $ \mu $, especially $\nu \in \mathcal{B}_{\nu(\mathbb{R}^d)}(\mu)_{(p, q)} $ for all probability measures $ \mu $, and by Proposition \ref{3.2}, if $ \nu=\sum_{\lambda \in \Lambda} \delta_\lambda $ is a finite discrete $ (p,q) $-Bessel measure for $ \mu $, then there exists a finite $ q $-Bessel for $ L^p(\mu)$. In the following we give an example of a discrete spectral measure $ \mu $ such that it has a finite discrete $ (p,q) $-Bessel measure $ \nu $ with Bessel bound less than $ \nu(\mathbb{R}^d) $, precisely $ \nu \in \mathcal{B}_{1}(\mu)_{(p, q)} $, where $1\leq p \leq 2 $ and $ q $ is the conjugate exponent to $ p $.
\begin{example} \label{3.21} Consider the atomic measure $ \mu :=\frac{1}{2} (\delta_0 +\delta_\frac{1}{2}) $, the set $ \{e_l :\l \in L:=\{0,1\} \} $ is an orthonormal basis for $ L^2 (\mu) $. Hence $ \sum_{l \in L } |\left<f,e_l \right>_{L^2(\mu)}|^2 = \| f\|_{L^2(\mu)}^2 $  for all $ f\in L^2(\mu) $. Moreover, for every $ f $ in $ L^1(\mu) $ we have $ \|\{[f, e_l]_{L^1(\mu)}\}_{l\in L} \|_\infty \leq \| f \|_{L^1(\mu)} $. Now by applying the Riesz-Thorin interpolation theorem $ \sum_{l \in L } |[f,e_l]_{L^p(\mu)}|^q \leq \| f\|_{L^p(\mu)}^q $, for all $ f\in L^p(\mu) $, where $1\leq p \leq 2 $ and $ q $ is the conjugate exponent to $ p $. Therefore, $\{e_l\}_{l\in L} $ is a finite $ q $-Bessel for $ L^p(\mu)$ with bound 1, and $ \nu=\sum_{l \in L} \delta_l $ is a finite discrete $ (p,q) $-Bessel measure for $ \mu $, especially $ \nu \in \mathcal{B}_{1}(\mu)_{(p, q)} $, where $1\leq p \leq 2 $ and $ q $ is the conjugate exponent to $ p $. When $ p > 2 $ and $ q $ is the conjugate exponent to $ p $, based on Proposition \ref{3.3} $ \nu \in \mathcal{B}_{2}(\mu)_{(p, q)} $ and $\{e_l\}_{l\in L} $ is a finite $ q $-Bessel for $ L^p(\mu)$ with bound 2. 
\end{example}
\begin{proposition}[\cite{13}] \label{3.22}
Let $ \mu= \sum_{c\in C}p_c\delta_c $ be a discrete probability measure on $ \mathbb{R}^d $. $ \mu $ is an F-spectral measure with an F-spectrum $ \Lambda $ if and only if $ \#C<\infty $ and $ \#\Lambda<\infty $.
\end{proposition} 
\begin{corollary}
Let $ 1<p, q<\infty $ and $ \dfrac{1}{p} + \dfrac{1}{q} =1 $. If $ \mu $ is any probability measure, then we have the following two equivalent statements:

(i) A finite discrete measure $ \nu=\sum_{\lambda \in \Lambda} \delta_\lambda $ is a $ (p,q) $-Bessel measure for $ \mu $, precisely $\nu \in \mathcal{B}_{\nu(\mathbb{R}^d)}(\mu)_{(p, q)} $. If $ \mu= \sum_{c\in C}p_c\delta_c $ is an F-spectral measure, then there may exist a better $ (p,q) $-Bessel bound for $ \nu $, where $1< p \leq 2 $ and $ q $ is the conjugate exponent to $ p $. 

(ii) A finite sequence $ \{e_\lambda\}_{\lambda \in \Lambda} $ is a $ q $-Bessel for $ L^p(\mu) $ with bound $ \nu(\mathbb{R}^d) $. If $ \mu= \sum_{c\in C}p_c\delta_c $ is an F-spectral measure, then $ \{e_\lambda\}_{\lambda \in \Lambda} $ may admit a better $ q $-Bessel bound, where $1< p \leq 2 $ and $ q $ is the conjugate exponent to $ p $. 
\end{corollary}
\begin{proof}
 The conclusion follows from Propositions \ref{3.3}, \ref{3.22}, \ref{3.2}, and the Riesz-Thorin interpolation theorem (see also Example \ref{3.21}).
\end{proof} 
\begin{theorem}[\cite{4}] \label{t.0}
Let $R$ be a $d \times d$ expansive integer matrix, $0 \in \mathcal{A} \subset \mathbb{Z}^d$. Let $ \mu_\mathcal{A} $ be an invariant measure associated to the iterated function system
\begin{equation*}
\tau_a (x) = R^{-1}(x + a) \quad (x \in \mathbb{R}^d, a \in\mathcal{A})
\end{equation*}
and the probabilities $(\rho_a)_{a\in \mathcal{A}}$. Then $ \mu $ has an infinite B-spectrum of positive Beurling dimension (Beurling dimension is used as a method of investigating existence of Bessel spectra for singular measures).
\end{theorem}
\begin{theorem} \label{t.1}
 Any fractal measure $ \mu $ obtained from an affine iterated function system has an infinite discrete $(p,q)$-Bessel measure $ \nu $, where $1\leq p \leq 2 $ and $ q $ is the conjugate exponent to $ p $.
\end{theorem}
\begin{proof}
Suppose that $R$ is a $d \times d$ expansive integer matrix, $0 \in \mathcal{A} \subset \mathbb{Z}^d$. If $ \mu_\mathcal{A} $ is an invariant measure associated to the iterated function system
\begin{equation*}
\tau_a (x) = R^{-1}(x + a) \quad (x \in \mathbb{R}^d, a \in\mathcal{A})
\end{equation*}
and the probabilities $(\rho_a)_{a\in \mathcal{A}}$, then according to Theorem \ref{t.0} there exist an infinite subset $ \Lambda $ of $ \mathbb{R}^d $ and a constant $ B > 0 $ such that 
\begin{equation*}
 \sum_{\lambda \in \Lambda }|\left<f,e_\lambda \right>_{L^2(\mu_\mathcal{A})}|^2 \leq B\| f\|_{L^2(\mu_\mathcal{A})}^2 \quad \text{for all }  f\in L^2(\mu_\mathcal{A}). 
\end{equation*} 
We also have $ \| \{[f,e_\lambda]_{L^1(\mu_\mathcal{A})}\}_{\lambda \in \Lambda} \|_\infty \leq \| f \|_{L^1(\mu_\mathcal{A})} $, for every $ f \in L^1(\mu_\mathcal{A}) $.
Now if $ \dfrac{1}{p} = 1 -\dfrac{\theta}{2} ; \ \dfrac{1}{q} =\dfrac{\theta}{2}$, for $ 0 <\theta < 1 $ (i.e., $1< p < 2 $ and $ q $ is the conjugate exponent to $ p $), then the Riesz-Thorin interpolation theorem yields
\begin{equation*}
 \sum_{\lambda \in \Lambda }|[f,e_\lambda]_{L^p(\mu_\mathcal{A})}|^q \leq B'^q\| f\|_{L^p(\mu_\mathcal{A})}^q \quad \text{for all }  f\in L^p(\mu_\mathcal{A}), 
\end{equation*}
where $ B'\leq B^{\frac{1}{2}\theta} $. Thus, $ \nu=\sum_{\lambda \in \Lambda} \delta_\lambda $ is a $(p,q)$-Bessel measure for $ \mu_\mathcal{A} $, and $ \nu \in \mathcal{B}_{B'^q}(\mu_\mathcal{A})_{(p, q)} $, where $1 < p < 2 $ and $ q $ is the conjugate exponent to $ p $. Moreover, we have $ \nu \in \mathcal{B}_{B}(\mu_\mathcal{A})_{(2, 2)} $ and $ \nu \in \mathcal{B}_{1}(\mu_\mathcal{A})_{(1, \infty)} $. On the other hand for every $1 < p < 2 $ and $ q $ (the conjugate exponent to $ p $), $ \{e_\lambda\}_{\lambda \in \Lambda} $ is a $ q $-Bessel for $ L^p(\mu_\mathcal{A}) $ with bound $ B'^q $.
\end{proof}
If a measure $ \mu $ is an F-spectral measure, then it must be of pure type, i.e, $ \mu $ is either discrete, singular continuous or absolutely continuous \cite{19, 13}. The case when the measure $ \mu $ is singular continuous, is not precisely known. The first known example of a singular continuous spectral measure supported on a non-integer dimension set (a fractal measure), was given by Jorgensen and Pedersen \cite{16}. They showed that the measure $ \mu_4 $ (the Cantor measures supported on Cantor set of $1/4$ contraction), is spectral. A spectrum of $\mu_4 $ is
$ \Lambda =\left \{ \sum_{m=0}^k 4^m d_m \; :  d_m \in \{0, 1\}, k\in \mathbb{N} \right\} $. They also showed that $ \mu_{2k} $ (the Cantor measures with even contraction ratio) is spectral, but $ \mu_{2k +1} $ (the Cantor measures with odd contraction ratio) is not.
\begin{remark}
 Since Cantor type measures are fractal measures, by applying Theorem~\ref {t.1} one can obtain that every Cantor type measure $ \mu $ admits a $ (p,q) $-Bessel measure $ \nu=\sum_{\lambda \in \Lambda_{ \mu } } \delta_\lambda $, where $1\leq p \leq 2 $ and $ q $ is the conjugate exponent to $ p $. Moreover, for every spectral Cantor type measure $ \mu_{2k} $, we have $ \nu \in \mathcal{B}_1(\mu_{2k})_{p, q}$, where $1\leq p \leq 2 $ and $ q $ is the conjugate exponent to $ p $.   
\end{remark}
In \cite{21} the author presents a method for constructing many examples of continuous measures $ \mu $ (including fractal ones) which have
components of different dimensions, but nevertheless they are F-spectral measures. In the following we give some results from \cite{21}. By applying the Riesz-Thorin interpolation theorem, one can obtain infinite discrete $ (p,q) $-Bessel measures $ \nu=\sum_{\lambda \in \Lambda_{\mu}} \delta_\lambda $ (where $1\leq p \leq 2 $ and $ q $ is the conjugate exponent to $ p $), for such F-spectral measures $ \mu $.
\begin{definition}[\cite{21}] Let $ \mu $ and $ \mu' $ be positive and finite measures on $ \mathbb{R}^n $ and $ \mathbb{R}^m $, respectively. A \emph{mixed type measure} $ \rho $ is a measure which is constructed on $ \mathbb{R}^{n+m} = \mathbb{R}^n \times \mathbb{R}^m $ and defined by
\begin{equation*}
\rho = \mu \times \delta_0 + \delta_0 \times \mu', 
\end{equation*}
where $ \delta_0 $ denotes the Dirac measure at the origin. Equivalently, the measure $ \rho $ may be defined by the requirement that 
\begin{equation*}
\int _{\mathbb{R}^n \times \mathbb{R}^m} f(x, y) d\rho (x, y) = \int _{\mathbb{R}^n} f(x, 0) d\mu(x) + \int _{\mathbb{R}^m} f(0, y) d\mu'(y),
\end{equation*}
for every continuous, compactly supported function $ f $ on $ \mathbb{R}^n \times \mathbb{R}^m $.
\end{definition}
\begin{theorem}[\cite{21}]
Let $ \mu $ and $ \mu' $ be continuous F-spectral measures. Then the mixed type measure $ \rho = \mu \times \delta_0 + \delta_0 \times \mu' $ is also an F-spectral measure.
\end{theorem}

\begin{theorem}[\cite{21}]
 If $ \mu $ is the sum of the $ k $-dimensional area measure on $[0, 1]^k \times \{0\}^{d-k}$, and the\\
  $j$-dimensional area measure on $\{0\}^{d-j} \times [0, 1]^j$ where $1 \leq j, k \leq d - 1$, then $ \mu $ is an F-spectral measure.
\end{theorem}
The following theorem provides many examples of single dimensional measures which are F-spectral measures:
\begin{theorem}[\cite{21}]
Let $ \phi: \mathbb{R}^k \rightarrow \mathbb{R}^{d-k} $ be a smooth function $(1 \leq k \leq d - 1)$. If $ \mu $ is the k-dimensional area measure on a compact subset of the graph $\{(x, \phi(x))  : x \in \mathbb{R}^k\}$ of $ \phi $, then $ \mu $ is an F-spectral measure.
\end{theorem}
The next proposition shows that if $ 1<p, q<\infty $ and $ \dfrac{1}{p} + \dfrac{1}{q} =1 $, then considering any countable subset (finite or infinite) $ \Lambda $ of $ \mathbb{R}^d $, one can obtain tight $ (p,q) $-frame measures and $ (p,q) $-Plancherel measures $\nu_\Lambda $ for $ \delta_0 $. In addition, there exist tight and Parseval $ q $-frames for $ L^p(\delta_0) $.
 \begin{proposition} \label{3.20}
Suppose that $ 1<p, q<\infty $ and $ \dfrac{1}{p} + \dfrac{1}{q} =1 $. Then there exists a measure $ \mu $ which admits tight $ (p,q) $-frame measures and $ (p,q) $-Plancherel measures. Moreover, there exist tight and Parseval $ q $-frames for $ L^p(\mu) $.
 \end{proposition}
 \begin{proof}
 Let $ \mu=\delta_0 $. For a countable subset $ \Lambda $ of $ \mathbb{R}^d $, Let $\nu_\Lambda =\sum_{\lambda\in\Lambda} c_\lambda \delta_{\lambda} $ where $ c_\lambda > 0 $.
  
 If $ \sum_{\lambda\in\Lambda} c_\lambda = m \neq 1 $, then for all $ f\in L^p(\mu) $,
 \begin{equation*}
 \int_{\mathbb{R}^d} |[f, e_t]_{L^p(\mu)}|^q d\nu(t) = \sum_{\lambda\in\Lambda} c_\lambda |f(0)|^q = m\| f \|_{L^p(\mu)}^q.
 \end{equation*}  
 
 If $ 0 < c_\lambda <1 $ and $ \sum_{\lambda\in\Lambda} c_\lambda = 1 $, then for all $ f\in L^p(\mu) $, 
 \begin{equation*}
 \int_{\mathbb{R}^d} |[f, e_t]_{L^p(\mu)}|^q d\nu(t) = \sum_{\lambda\in\Lambda} c_\lambda |f(0)|^q = \| f \|_{L^p(\mu)}^q.
 \end{equation*}
 
 On the other hand, for all $ f\in L^p(\mu) $ we have
   \begin{equation*}
\int_{\mathbb{R}^d} |[f, e_t]_{L^p(\mu)}|^q d\nu(t) = \sum_{\lambda\in\Lambda} c_\lambda |[f, e_{\lambda}]_{L^p(\mu)}|^q = \sum_{\lambda\in\Lambda} |[f, \sqrt [q]c_\lambda e_{\lambda}]_{L^p(\mu)}|^q.
 \end{equation*}
Hence If $ \sum_{\lambda\in\Lambda} c_\lambda = m \neq 1 $, then $ \{ \sqrt [q]c_\lambda e_{\lambda}\}_{\lambda \in \Lambda} $ is a tight $ q $-frame for $ L^p(\mu) $, and If $ 0 < c_\lambda <1 $, $ \sum_{\lambda\in\Lambda} c_\lambda = 1 $, then $ \{ \sqrt [q]c_\lambda e_{\lambda}\}_{\lambda \in \Lambda} $ is a Parseval $ q $-frame for $ L^p(\mu) $.
 \end{proof}
\begin{proposition}
 Let $ \mu $ be a finite Borel measure and $ B $ be a positive constant. Then there exists a $ (p,q) $-Bessel measure $ \nu $ for $ \mu $ for all $ 1< p, q < \infty $ and $ \dfrac{1}{p} + \dfrac{1}{q} =1 $, such that $ \nu \in \mathcal{B}_B(\mu)_{p,q}  $. In addition, for every $ 1< p, q < \infty $ and $ \dfrac{1}{p} + \dfrac{1}{q} =1 $, there exists a $ q $-Bessel with bound $ B $ for $ L^p(\mu) $.
\end{proposition}
\begin{proof} Let $ \nu =\sum_{i\in I}c_i \delta_{\lambda_i} $ for some $ \lambda_i\in \mathbb{R}^d $ such that $\sum_{i\in I}c_i \leq \dfrac{B}{\mu(\mathbb{R}^d)} $. Let $ p>1 $ and $ f\in L^p(\mu) $. If $ q $ is the conjugate exponent to $ p $, then by applying Holder's inequality we have
\begin{equation}\label{4}
\int_{\mathbb{R}^d} |[f, e_t]_{L^p(\mu)}|^q d\nu(t) =\sum_{i\in I}c_i\int_{\mathbb{R}^d} |[f, e_t]_{L^p(\mu)}|^q d\delta_{\lambda_i}(t)\leq\sum_{i\in I}c_i \parallel f\parallel_{L^p(\mu)}^q \mu(\mathbb{R}^d) \leq B\parallel f\parallel_{L^p(\mu)}^q.
\end{equation}
 Hence  $\nu\in\mathcal{B}_B(\mu)_{p,q} $.
 
Since
\begin{equation*}
\sum_{i\in I} |[f, \sqrt [q]c_i e_{\lambda_i}]_{L^p(\mu)}|^q = \sum_{i\in I}c_i |[f, e_{\lambda_i}]_{L^p(\mu)}|^q = \int_{\mathbb{R}^d} |[f, e_t]_{L^p(\mu)}|^q d\nu(t),
\end{equation*}
the second statement follows from (\ref{4}).
\end{proof}
All infinite $ (p,q) $-Bessel measures $ \nu $ we observed were discrete. Now the question is whether we can find a finite measure $ \mu $ which admits a continuous infinite $ (p,q) $-Bessel measure $ \nu $. The following proposition shows that the answer is affirmative. 
\begin{proposition}\label{3.26}
If $ \nu=\lambda $ is the Lebesgue measure on $ \mathbb{R}^d $ and $ \mu=\lambda |_{[0 , 1]^d} $, then $ \lambda $ is a $ (p,q) $-Bessel measure for $ \mu $ where $1\leq p \leq 2 $ and $ q $ is the conjugate exponent to $ p $.
\end{proposition}
\begin{proof}
According to Plancherel's theorem the following equation is satisfied:
\begin{equation*}
\int_{\mathbb{R}^d} |\hat{f}(t)|^2 d\lambda(t) =\int_{\mathbb{R}^d} |f(x)|^2 d\lambda(x) \quad \text{for all }  f\in L^2(\lambda).
\end{equation*} 
 If $ f $ is supported on $ [0 , 1]^d $, then 
 \begin{equation*}
\int_{\mathbb{R}^d} |\widehat{f d\mu}|^2 d\lambda(t) =\int_{[0, 1]^d} |f(x)|^2 d\mu (x) \quad \text{for all }  f\in L^2(\mu).
\end{equation*} 
Moreover, we have $ \| \widehat{fd\mu} \|_\infty \leq \| f \|_{L^1(\mu)} $ for all $ f $ in $ L^1(\mu) $.  Now by applying the Riesz-Thorin interpolation theorem
\begin{equation*}
 \int_{\mathbb{R}^d} |\widehat{f d\mu}|^q d\lambda(t) \leq \| f\|_{L^p(\mu)}^q \quad \text{for all }  f\in L^p(\mu), 
\end{equation*}
where $1\leq p \leq 2 $ and $ q $ is the conjugate exponent to $ p $. Hence $ \lambda \in \mathcal{B}_1(\mu)_{p,q} $.
\end{proof}
\begin{corollary} The measure $ \mu=\lambda |_{[0 , 1]^d} $ has infinite continuous and discrete $ (p,q) $-Bessel measures, where $1\leq p \leq 2 $ and $ q $ is the conjugate exponent to $ p $. More precisely, if $ \nu_1 = \sum_{t\in \mathbb{Z}^d}\delta_t $ and $ \nu_2=\lambda $, then $ \nu_1, \nu_2 \in \mathcal{B}_1(\mu)_{p,q} $.
\end{corollary} 
\begin{proof}
The conclusion follows from Example \ref{ex 4} and Proposition \ref{3.26}. 
\end{proof}
\section{Properties and Structural Results}
In this section our assertions are based on the results by Dutkay, Han, and Weber from \cite{5}. We generalize the results and give some of the proofs for completeness.
\begin{proposition}\label{prop 2}
Let $ \mu $ be a Borel probability measure. Let $ 1< p, q < \infty $ and $ \dfrac{1}{p} + \dfrac{1}{q} =1 $. If $\nu$ is a $(p,q)$-Bessel measure for $\mu$, then there exists a constant $ C $ such that $ \nu(K) \leq C diam(K)^d $ for any compact subset $ K $ of $ \mathbb{R}^d $. Accordingly, $ \nu $ is $\sigma$-finite.
\end{proposition}
\begin{proof}
It is easy to check that $\widehat{d\mu} : \mathbb{R}^d \rightarrow \mathbb{C}$ is uniformly continuous and $\widehat{d\mu}(0) =\mu(\mathbb{R}^d) =1$. So for every $ \eta>0 $ there exists $\epsilon > 0$ such that for $x\in \mathsf{B}(0,\epsilon)$ we have $|\widehat{d\mu}(0)|-|\widehat{d\mu}(x)|\leq|\widehat{d\mu}(0)-\widehat{d\mu}(x)|\leq \eta$, and then $ |\widehat{d\mu}(x)|\geq 1-\eta $. If $ \delta:= (1-\eta)^q $, then
$|\widehat{d\mu}(x)|^q \geq \delta$ for $x\in \mathsf{B}(0,\epsilon)$. Thus, for any $t \in \mathbb{R}^d$, 
\begin{align*}
B=B\|e_t\|_{L^p(\mu)}^q &\geq \int_{\mathbb{R}^d} |[e_t , e_x]|^q d\nu(x) =\int_{\mathbb{R}^d} |[1 , e_{x-t}]|^q d\nu(x)= \int_{\mathbb{R}^d} |\widehat{d\mu}(x-t)|^q d\nu(x)\\
 &\geq \int_{\mathsf{B}(t,\epsilon)}|\widehat{d\mu} (x-t)|^q d\nu(x) \geq \nu(\mathsf{B}(t,\epsilon))\delta.
 \end{align*}
Now Let $ K\subseteq\mathbb{R}^d $ be compact and $ r= diam (K) $. Then there exists a point $ x = (x_1, \ldots, x_d) $ in $ \mathbb{R}^d $ such that $ K \subset \prod_{i=1}^d [x_i-r, x_i+r] $. We may assume that $ \epsilon < 2r $ and $ 2r/\epsilon \in \mathbb{N} $. Let $ M= 2r/\epsilon $. We have $ \prod_{i=1}^d [x_i-r, x_i+r] = \bigcup_{\alpha=1}^{M^d} C_\alpha $ where $ C_\alpha $s are d-dimensional cubes of side length $ \epsilon $. For any $ \alpha \in \{1,\ldots, M^d \} $, let $ t_\alpha $ be the center point of $ C_\alpha $. Then $ C_\alpha \subset \mathsf{B}(t_\alpha,\epsilon) $. Now if $ C:= (2/\epsilon)^d B/\delta $, then
\begin{equation*}
\nu(K) \leq \nu \left(\bigcup_{\alpha=1}^{M^d}\mathsf{B}(t_\alpha, \epsilon)\right) \leq\sum^{M^d}_{\alpha=1}\nu(\mathsf{B}(t_\alpha, \epsilon) )\leq \left(\dfrac{2r}{\epsilon}\right)^d \dfrac{B}{\delta} = r^d \left(\dfrac{2}{\epsilon}\right)^d \dfrac{B}{\delta} = C r^d.
\end{equation*}
Hence the assertion follows.
\end{proof}
\begin{theorem}
Let $ 1< p, q < \infty $ and $ \dfrac{1}{p} + \dfrac{1}{q} =1 $. Let $ B > A > 0 $. Then the set $\mathcal{F}_{A,B}(\mu)_{p,q}$ is empty for some finite compactly supported Borel measures $ \mu $.
\end{theorem}
\begin{proof}
Let $\mu=\chi_{[0,1]} dx+\delta_2$. Suppose $\nu \in \mathcal{F}_{A,B}(\mu)_{p,q}$. Let $f :=\chi_{\{2\}}$. Then $\| f\|_{L^p (\mu)} = 1$ and $|[f,e_t]| = 1$ for all $ t\in\mathbb{R}$. In addition, the upper bound implies that $\nu(\mathbb{R}) \leq B < \infty$. Then from the inner regularity of Borel measures we obtain that for any $\epsilon > 0$ there exist a compact set $ K\subset \mathbb{R} $ and a positive constant $ R $ such that $ \nu(\mathbb{R})- \epsilon < K \leq \nu (\mathsf{B}(0, R)) $. Therefore $\nu(\mathbb{R} \setminus \mathsf{B}(0, R)) < \epsilon$.

 Choose some $T$ large, arbitrary and let $g(x) := e^{-2\pi iTx}\chi_{[0,1]}$. Then 
\begin{equation*}
|[g,e_t]_{L^p (\mu)}|^q = \left|\int_{[0, 1]} e^{-2\pi i (T + t)x} dx \right|^q = \left|\frac{sin(\pi(T+t))}{\pi(T+t)}\right|^q    \qquad (t\in\mathbb{R}).
\end{equation*}
The substitution $ z:= -2\pi x $ gives the last equality.
 Consequently, $|[g,e_t]_{L^p(\mu)}|^q \leq 1$ for all $ t\in\mathbb{R}$, and if we take $T\geq2R$, then  for all $t\in (-R,R)$
\begin{equation*}
|[g,e_t]_{L^p(\mu)}|^q \leq \frac{1}{\pi^q(T-R)^q}.  
\end{equation*}
Hence from the lower bound we obtain
\begin{align*}
A=A\| g\|_{L^p(\mu)}^q & \leq \int_\mathbb{R} |[g,e_t]_{L^p(\mu)}|^q d\nu(t) = \int_{\mathsf{B}(0, R)}|[g,e_t]_{L^p(\mu)}|^q d\nu(t) +\int_{\mathbb{R} \setminus \mathsf{B}(0, R)}|[g,e_t]_{L^p(\mu)}|^q d\nu(t)\\&\leq \frac{1}{\pi^q(T-R)^q}.\nu(\mathbb{R}) + \epsilon.
\end{align*}
Now if $T \rightarrow \infty$ and $\epsilon \rightarrow 0$, then $ A=0 $. This is a contradiction.
\end{proof}
The next proposition shows that if there exists a $(p,q)$-Bessel/frame measure, then many others can be constructed.
\begin{proposition} \label{prop 4.3}
Let $ \mu $ be a finite Borel measure and $ A, B $ be positive constants. Let $ 1< p, q < \infty $ and $ \dfrac{1}{p} + \dfrac{1}{q} =1 $. Then both sets $ \mathcal{B}_B(\mu)_{p,q} $ and $ \mathcal{F}_{A,B}(\mu)_{p,q} $ are convex and closed under convolution with Borel probability measures.  
\end{proposition}
\begin{proof}
Let $ \nu_1 , \nu_2 \in\mathcal{B}_B(\mu)_{p,q} $ and $0 <\lambda <1$. For all $f \in L^p (\mu)$, 
\begin{equation*}
\int_{\mathbb{R}^d} |\widehat{fd\mu}|^q d(\lambda\nu_1 +(1-\lambda)\nu_2)=\lambda\int_{\mathbb{R}^d} |\widehat{fd\mu}|^q d\nu_1 + (1-\lambda)\int_{\mathbb{R}^d} |\widehat{fd\mu}|^q d\nu_2 \leq B\| f\|_{L^p(\mu)}^q.
\end{equation*}
Then $ \lambda\nu_1 +(1-\lambda)\nu_2 \in\mathcal{B}_B(\mu)_{p,q} $. Similarly, if $ \nu_1 , \nu_2 \in\mathcal{F}_{A,B}(\mu)_{p,q}$, then $ \lambda\nu_1 +(1-\lambda)\nu_2 \in \mathcal{F}_{A,B}(\mu)_{p,q}$.

Let $ s\in \mathbb{R}^d $. Then for all $f \in L^p (\mu)$,
\begin{equation*}
\|e_s f\|^p_{L^p(\mu)}=\int_{\mathbb{R}^d} |e_s(x)f(x)|^p d\mu(x)=\int_{\mathbb{R}^d} |f(x)|^p d\mu(x)=\|f\|^p_{L^p(\mu)}.
\end{equation*}
In addition, Let $ \nu\in\mathcal{B}_{B}(\mu)_{p,q}  $ and $ \rho $ be a Borel probability measure on $\mathbb{R}^d$. Then for any $ t \in \mathbb{R}^d $ and $f \in L^p (\mu)$,  
\begin{align*}
[e_{-s}f, e_t]_{L^p(\mu)}&=\int_{\mathbb{R}^d} e_{-s}(x)f(x)e^{-2\pi it\cdot x} d\mu(x) =\int_{\mathbb{R}^d}f(x)e^{-2\pi i(s+t)\cdot x} d\mu(x) \\
&=[f ,e_{s+t}]_{L^p(\mu)}.
\end{align*}
Therefore, 
\begin{align*}
\int_{\mathbb{R}^d} |[f,e_t]_{L^p(\mu)}|^q d\nu \ast \rho(t) & = \int_{\mathbb{R}^d} \int_{\mathbb{R}^d} |[f,e_{t+s}]_{L^p(\mu)}|^q d\nu(t) \ d\rho(s) = \int_{\mathbb{R}^d} \int_{\mathbb{R}^d} |[e_{-s}f,e_t]_{L^p(\mu)}|^q d\nu(t) \ d\rho(s)\\&\leq \int_{\mathbb{R}^d} B\| e_{-s}f\|_{L^p(\mu)}^q d\rho(s) = B \int_{\mathbb{R}^d}\| f\|_{L^p(\mu)}^q d\rho(s) = B\| f\|_{L^p(\mu)}^q.
\end{align*}
 For $ \nu\in\mathcal{F}_{A,B}(\mu)_{p,q} $ one can obtain the lower bound analogously. 
\end{proof}
\begin{corollary}
Let $ 1< p, q < \infty $ and $ \dfrac{1}{p} + \dfrac{1}{q} =1 $. If there exists a $(p,q)$-Bessel/frame measure for $\mu$, then there exists one which is absolutely continuous with respect to the Lebesgue measure and whose Radon-Nikodym derivative is $ C^\infty $.
\end{corollary}
\begin{proof}
Let $ \nu $ be a $(p,q)$-Bessel/frame measure for $\mu$. Convoluting $ \nu $ with the Lebesgue measure on $ [0,1] $ we have
\begin{align*}
\nu*\chi_{[0 ,1]}dm (E)&=\int_{\mathbb{R}}\int_{\mathbb{R}} \chi_E (x+y)d\nu(x) \chi_{[0 ,1]}(y) dm (y) =\int_{\mathbb{R}}\int_{\mathbb{R}} \chi_E (t) \chi_{[0 ,1]}(t-x) d\nu(x)dm (t-x)\\
&=\int_{\mathbb{R}}\chi_E (t) \nu([t-1 , t])dm(t) =\int_E  \nu([t-1 , t])dm,
\end{align*}
where $ E $ is any Borel subset of $ \mathbb{R} $. Thus, we obtained a $(p,q)$-Bessel/frame measure for $ \mu $ which is absolutely continuous with respect to the Lebesgue measure.
 
Now consider the following two propositions from \cite{10}.\\  
$ (i) $ If $ d\mu= fdm $ and $ d\nu= gdm $, then $ d(\mu\ast \nu) = (f \ast g) dm $.\\
$ (ii) $ If $ f\in L^1 $ (or $ f $ is locally integrable on $ \mathbb{R}^d $), $ g\in C^k $, and $ \partial^\alpha $ is bounded for $ |\alpha|\leq k $, then $ f \ast g \in C^k  $ and $ \partial^\alpha (f \ast g) = f \ast(\partial^\alpha g) $ for $ |\alpha|\leq k $.

 Let $ g \geq 0 $ be a compactly supported $ C^\infty $-function with $ \int g(t) dt = 1 $. Let $ d\mu_0 = gdm $ and \\
 $ d\nu_0 = \nu \ast \chi_{[0 ,1]}dm $. Then we have $ d(\mu_0\ast \nu_0) = ( \nu([\cdot-1 , \cdot]) \ast g) dm $ and $\nu([\cdot-1 , \cdot]) \ast g \in C^\infty $.
\end{proof}
\begin{definition}[\cite{5}]
A sequence of Borel probability measures \{$ \lambda_n\} $ is called an \emph{approximate identity} if 
\begin{equation*}
 sup\{\parallel t\parallel : t\in supp \lambda_n\} \rightarrow 0  \quad \text{as}\quad  n\rightarrow \infty .
\end{equation*} 
\end{definition}
\begin{lemma}[\cite{5}]\label{lem 6}
Let $ \{\lambda_n\} $ be an approximate identity. If $ f $ is a continuous function on $\mathbb{R}^d $, then for any $ x\in\mathbb{R}^d $,\;$\int f(x+t)\ d\lambda_n(t) \rightarrow f(x)$  as  $n \rightarrow \infty$.
\end{lemma}
By Proposition \ref{prop 4.3}, if $ \nu $ is a $(p,q)$-Bessel/frame measure for $ \mu $, then $ \nu \ast \lambda $ is a $(p,q)$-Bessel/frame measure for $ \mu $ with the same bound(s), where $ \lambda $ is any Borel probability measure. An obvious question is under what conditions the converse is true. The next theorem gives an answer.
\begin{theorem}
Let $ 1< p, q < \infty $ and $ \dfrac{1}{p} + \dfrac{1}{q} =1 $. Let $ \{\lambda_n\} $ be an approximate identity. Suppose $ \nu $ is a $ \sigma $-finite Borel measure, and suppose all measures $ \nu \ast \lambda_n $ are $(p,q)$-Bessel/frame measures for $ \mu $ with uniform bounds, independent of $ n $. Then $ \nu $ is a $(p,q)$-Bessel/frame measure.
\end{theorem}
\begin{proof}
Take $ f\in L^p(\mu) $. Since $| [f, e_.]_{L^p(\mu)}|^q $ is continuous on $\mathbb{R}^d $, by Lemma \ref{lem 6} and Fatou's lemma we have
 \begin{align*}
\int_{\mathbb{R}^d} |[f,e_x]_{L^p(\mu)}|^q d\nu(x) & \leq \liminf_n \int_{\mathbb{R}^d} \int_{\mathbb{R}^d} |[f,e_{x+t}]_{L^p(\mu)}|^q d\lambda_n(t)\ d\nu(x)\\
&= \liminf_n \int_{\mathbb{R}^d} |[f,e_y]_{L^p(\mu)}|^q d(\nu \ast \lambda_n)(y)\\
&\leq B\| f \|_{L^p(\mu)}^q. 
\end{align*}
Hence $ \nu $ is a $(p,q)$-Bessel measure with the same bound $ B $ as $ \nu \ast \lambda_n $.

Now showing that 
\begin{equation*}
\int_{\mathbb{R}^d} |[f,e_x]_{L^p(\mu)}|^q d(\nu \ast \lambda_n) \rightarrow \int_{\mathbb{R}^d} |[f,e_x]_{L^p(\mu)}|^q d\nu,
\end{equation*}
gives the lower bound (see \cite{5}). 
\end{proof}
We need the following two propositions from \cite{5} to present a general way of constructing $ (p,q) $-Bessel/frame measures for a given measure. 
\begin{proposition}[\cite{5}]
Let $ \mu $ and $ \mu' $ be Borel probability measures. For $ f\in L^1(\mu) $, the measure $ (f d\mu)\ast \mu' $ is absolutely continuous w.r.t. $ \mu \ast \mu' $ and if the Radon-Nikodym derivative is denoted by $ P f $, then
\begin{equation*}
P f = \frac{(f d\mu)\ast\mu'}{d(\mu\ast\mu')}.
\end{equation*}
\end{proposition}
\begin{proposition}[\cite{5}]
Let $ \mu$, $ \mu' $ be two Borel probability measures and $ 1 \leq p\leq\infty $. if $ f\in L^p(\mu) $, then the function $ P f $ is also in $ L^p( \mu \ast \mu') $ and \begin{equation*}
\| P f\|_{ L^p( \mu \ast \mu')}\  \leq\  \| f\|_{ L^p( \mu )}.
\end{equation*}
\end{proposition}
Now we show that if a convolution of two measures admits a $ (p,q) $-Bessel/frame measure, then one can obtain a $ (p,q) $-Bessel/frame measure for one of the measures in the convolution by using the Fourier transform of the other measure in the convolution.
\begin{proposition}
Let $ \mu $, $ \mu' $ be two Borel probability measures. Let $ 1< p, q < \infty $ and $ \dfrac{1}{p} + \dfrac{1}{q} =1 $. If $ \nu $ is a $ (p,q) $-Bessel measure for $ \mu \ast \mu' $, then $ | \hat{\mu'} |^q d\nu $ is a $ (p,q) $-Bessel measure for $ \mu $ with the same bound.

If in addition $ \nu $ is a $ (p,q) $-frame measure for $ ( \mu \ast \mu') $ with bounds $ A $ and $ B $, and for all $ f \in L^p(\mu) $, $ c\| f \|_{L^p(\mu)}^q \leq \| P f\|_{ L^p( \mu \ast \mu')}^q $, then $ | \hat{\mu}' |^q d\nu $ is a $ (p,q) $-frame measure for $ \mu $ with bounds $ c A $ and $ B $.
 \end{proposition}  
 \begin{proof}
If $ \mu,\nu \in M(\mathbb{R}^d) $, then $ \widehat{\mu\ast\nu} = \hat{\mu}.\hat{\nu} $ (see\cite{10}). Take $ f \in L^p(\mu) $. Then
 \begin{equation*}
  \int_{\mathbb{R}^d} |\widehat{(f d\mu)}|^q . | \hat{\mu'} |^q d\nu = \int_{\mathbb{R}^d} |\widehat{(f d\mu)\ast \mu' |}^q d\nu = \int_{\mathbb{R}^d} |\widehat{P f d(\mu \ast \mu')}|^q d\nu.
 \end{equation*}
 Thus, we have
 \begin{equation*}
 cA\| f \|_{L^p(\mu)}^q\leq A \| P f\|_{ L^p( \mu \ast \mu')}^q \leq \int_{\mathbb{R}^d} |\widehat{P f d(\mu \ast \mu')}|^q d\nu \leq B \| P f\|_{ L^p( \mu \ast \mu')}^q \leq B\| f \|_{L^p(\mu)}^q.
 \end{equation*}
 \end{proof}
 In the next theorem we have some stability results. In fact, this theorem is a generalization of Proposition \ref{2.21}.
\begin{theorem}\label{theo 219}
Let $ \mu $ be a compactly supported Borel probability measure. Let $ 1< p, q < \infty $ and $ \dfrac{1}{p} + \dfrac{1}{q} =1 $. If $ \nu$ is a $ (p,q) $-Bessel measure for $ \mu $, then for any $ r > 0 $ there exists a constant $ D > 0 $ such that
\begin{equation*}
\int_{\mathbb{R}^d} \sup_{|y|\leq r} |[f,e_{x+y}]_{L^p(\mu)}|^q d\nu(x) \leq D \| f \|_{L^p(\mu)}^q, \qquad \text{for all}\; f \in L^p(\mu). 
\end{equation*}
\par If $ \nu $ is a $ (p,q) $-frame measure for $ \mu $, then there exist constants $ \delta > 0 $ and $ C > 0 $ such that 
\begin{equation*}
C\| f \|_{L^p(\mu)}^q \leq \int_{\mathbb{R}^d} \inf_{|y|\leq \delta} |[f,e_{x+y}]_{L^p(\mu)}|^q d\nu(x), \qquad \text{for all}\; f \in L^p(\mu).
\end{equation*}
\end{theorem}
\begin{proof}
The approach is completely similar to the proof of Theorem $ 2.10 $ from \cite{5}.
 \end{proof}
 
We show that by using this stability of $ (p,q) $-frame measures, one can obtain atomic $ (p,q) $-frame measures from a general $ (p,q) $-frame measure. 
\begin{definition}\label{4.12}
Let $ Q = [0,1)^d $ and $ r > 0 $. If $ \nu$ is a Borel measure on $ \mathbb{R}^d $ and if $ (x_k)_{k\in\mathbb{Z}^d} $ is a set of points such that for all $ k\in\mathbb{Z}^d $ we have $ x_k\in r(k +Q) $ and $ \nu(r(k +Q)) <\infty $, then a \emph{discretization of the measure $ \nu $} is defined by
\begin{equation*}
\nu' := \sum_{k\in\mathbb{Z}^d} \nu(r(k +Q))\delta_{x_k}.
\end{equation*}
\end{definition}
 \begin{theorem}\label{theo 220}
Let $ 1< p, q < \infty $ and $ \dfrac{1}{p} + \dfrac{1}{q} =1 $. If a compactly supported Borel probability measure $ \mu $ has a $(p,q)$-Bessel/frame measure $ \nu$, then it also has an atomic one. More precisely, if $ \nu$ is a $ (p,q) $-Bessel measure for $ \mu $ and if $ \nu' $ is a discretization of the measure $ \nu $, then $ \nu' $ is a $ (p,q) $-Bessel measure for $ \mu $.

If $ \nu $ is a $ (p,q) $-frame measure for $ \mu $ and $ r > 0 $ is small enough, then $ \nu' $ is a $ (p,q) $-frame measure for $ \mu $.
 \end{theorem}
\begin{proof} 
Let $ Q = [0,1)^d $. Let $ (x_k)_{k\in\mathbb{Z}^d} $ be a set of points such that $ x_k\in r(k +Q) $ for all $ k\in\mathbb{Z}^d $. For every $ x\in r(k +Q) $ define $ \epsilon(x) := x_k - x $. Thus, $ |\epsilon(x)| \leq r\sqrt{d} =:r' $ and for any $ f \in L^p(\mu) $,
\begin{align*} 
\int_{\mathbb{R}^d} |[f,e_{x +\epsilon(x)}]_{L^p(\mu)}|^q d\nu(x) &= \sum_{k\in\mathbb{Z}^d} \int_ {r(k +Q)} |[f,e_{x_k}]_{L^p(\mu)}|^q d\nu(x)\\
& = \sum_{k\in\mathbb{Z}^d}\nu(r(k +Q))|[f,e_{x_k}]_{L^p(\mu)}|^q.
\end{align*}

 Since we have
\begin{align*}
\int_{\mathbb{R}^d} \inf_{|y|\leq r'} |[f,e_{x+y}]_{L^p(\mu)}|^q d\nu(x)  & \leq \int_{\mathbb{R}^d} |[f,e_{x +\epsilon(x)}]_{L^p(\mu)}|^q d\nu(x)\\
 & \leq \int_{\mathbb{R}^d} \sup_{|y|\leq r} |[f,e_{x+y}]_{L^p(\mu)}|^q d\nu(x),
 \end{align*}
 the upper and lower bounds follow from Theorem \ref{theo 219}.   
\end{proof}  
By Lemma \ref{3.6}, if there exists a purely atomic $ (p,q) $-frame measure $ \nu $ for a probability measure $ \mu $, then there exists a $ q $-frame  for $ L^p(\mu) $. Now we conclude that if there exists a $ (p,q) $-frame measure $ \nu $ (not necessarily purely atomic) for a compactly supported probability measure $ \mu $, then there exists a $ q $-frame for $ L^p(\mu) $.              
 \begin{corollary}
Let $ \mu $ be a compactly supported Borel probability measure. Let $ 1< p, q < \infty $ and $ \dfrac{1}{p} + \dfrac{1}{q} =1 $. If $ \nu $ is a $ (p,q) $-frame measure for $ \mu $ with bounds $ A, B $ and $ r > 0 $ is sufficiently small, then there exist positive constants $ C, D $ such that $\{ c_ke_{x_k} : k\in\mathbb{Z}^d\} $ is a $ q $-frame for $ L^p(\mu) $ with bounds $ C, D $, where $ x_k\in r(k +Q) $ and $ c_k =\sqrt [q]{\nu(r(k +Q))} $. 
\end{corollary}  
\begin{proof}Let $ \nu \in \mathcal{F}_{A,B}(\mu)_{p, q} $. Then by Theorems \ref{theo 220} and \ref{theo 219},
 $ \nu'=\sum _{k\in\mathbb{Z}^d} c_k^q\delta_{x_k}$ is a $ (p,q) $-frame measure for $ \mu $. More precisely, $ \nu' \in\mathcal{F}_{C,D}(\mu)_{p, q} $. Hence for all $ f \in L^p(\mu) $,
\begin{equation*}
C\| f \|_{L^p(\mu)}^q\leq \int_{\mathbb{R}^d} |[f,e_t]_{L^p(\mu)}|^q d\nu'(t)=\sum_{k\in\mathbb{Z}^d}c_k^q |[f, e_{x_k}]_{L^p(\mu)}|^q= \sum_{k\in\mathbb{Z}^d} |[f,c_k e_{x_k}]_{L^p(\mu)}|^q\leq D\| f \|_{L^p(\mu)}^q.
\end{equation*}
\end{proof}
                                                                                                                                                                     
\section*{Acknowledgements}
The authors would like to thank Dr. Nasser Golestani for his valuable guidance and helpful comments.

\small $^{1}$Department of Mathematics , Science and Research Branch, Islamic Azad University, Tehran, Iran.

\emph{E-mail address}: \small{fz.farhadi61@yahoo.com} \\
 
 $^{2}$Department of Mathematics, Faculty of Science, Islamic Azad University, Central Tehran Branch,Tehran, Iran.
 
\emph{E-mail address}: \small{moh.asgari@iauctb.ac.ir}\\ 

$^{3}$Department of Mathematics , Science and Research Branch, Islamic Azad University, Tehran, Iran.

\emph{E-mail address}: \small{mrmardanbeigi@srbiau.ac.ir}\\ 

$^{4}$Department of Mathematics , Science and Research Branch, Islamic Azad University, Tehran, Iran. 

\emph{E-mail address}: \small{m.azhini@srbiau.ac.ir}

\end{document}